\documentclass[11pt,a4paper]{article}

\usepackage{amssymb,amsfonts,amsmath,amsthm,cite}
\usepackage{graphicx,epsfig,subfigure}
\usepackage{enumerate}
\usepackage{dsfont}
\usepackage{color,xcolor}
\allowdisplaybreaks[4]
\newtheorem{thm}{Theorem}[section]

\newtheorem{lem}{Lemma}[section]

\newtheorem{ex}{Example}[section]

\newtheorem{conj}{Conjecture}[section]

\makeatletter \@addtoreset{equation}{section}
\def\qed{\hfill \rule{4pt}{7pt}}
\def\hf{\mathcal{F}}

\def\ha{\mathcal{A}}
\def\hb{\mathcal{B}}

\def\ex{\mathcal{\mathbb{E}}}
\def\Var{\mathrm{Var}}

\begin{document}

\title{Extremal Problem for Matchings and Rainbow Matchings on Direct Products}

\author{Jian Wang$^1$, Jie You$^2$\\[5pt]
$^1$Department of Mathematics\\
Taiyuan University of Technology\\
Taiyuan 030024, P. R. China\\[6pt]
$^2$Center for Applied Mathematics\\
Tianjin University\\
Tianjin 300072, P. R. China\\[6pt]
E-mail:  $^1$wangjian01@tyut.edu.cn, $^2$yj\underline{ }math@tju.edu.cn
}
\date{}
\maketitle

\begin{abstract}
Let $n_1,\dots,n_\ell,k_1,\dots,k_\ell$ be integers and let $V_1,\dots,V_\ell$ be disjoint sets with $|V_i|=n_i$ for $i=1,\dots,\ell$. Define $\sqcup_{i=1}^\ell \binom{V_i}{k_i}$ as the collection of all subsets $F$ of $\cup_{i=1}^\ell V_i$ with $|F\cap V_i| =k_i$ for each $i=1,\dots,\ell$. In this paper, we show that if the matching number of  $\hf\subseteq \sqcup_{i=1}^\ell \binom{V_i}{k_i}$ is at most $s$ and $n_i\geq 4\ell^2 k_i^2s$ for all $i$, then $|\hf| \leq \max_{1\leq i\leq \ell}[\binom{n_i}{k_i}-\binom{n_i-s}{k_i}]\prod_{j\neq i}\binom{n_j}{k_j}$. Let $\hf_1,\hf_2,\dots,\hf_s\subseteq\sqcup_{i=1}^\ell \binom{V_i}{k_i}$ with $n_i\geq 8\ell^2k_i^2s$ for all $i$. We also prove that if $\hf_1,\hf_2,\dots,\hf_s$ are rainbow matching free, then there exists $t$ in $[s]$ such that
$|\hf_t|\leq  \max_{1\leq i\leq \ell}\left[\binom{n_i}{k_i}-\binom{n_i-s+1}{k_i}\right]\prod_{j\neq i}\binom{n_j}{k_j}.$
\end{abstract}
\section{Introduction}

Let $[n]=\{1,2,\dots,n\}$ and $\binom{[n]}{k}$ be the collection of all its $k$-subsets. For a $k$-graph $\hf\subset {[n]\choose k}$, the matching number of $\hf$, denote by $\nu(\hf)$, is the maximum number of pairwise disjoint members of $\hf$. Let $\hf_1, \hf_2,\dots, \hf_s\subseteq \binom{[n]}{k}$. We say that $\mathcal{F}_1, \mathcal{F}_2, \dots, \mathcal{F}_{s}$ contain a rainbow matching if there exist $s$ pairwise disjoint sets $F_1\in \mathcal{F}_1, F_2\in\mathcal{F}_2, \dots, F_s\in\mathcal{F}_s$.

One of the most important open problems in extremal set theory is to determine the maximum of $|\hf|$ subject to $\nu(\hf)\leq s$, which is known as the celebrated Erd\H{o}s matching conjecture.

\begin{conj}[Erd\H{o}s Matching Conjecture \cite{erdHos1965problem}]\label{conjEMC}
Let $n\geq (s+1)k$ and $\hf\subseteq\binom{[n]}{k}$. If $\nu(\hf)\leq s$, then
\begin{align}
|\hf|\leq \max\left\{\binom{k(s+1)-1}{k},\binom{n}{k}-\binom{n-s}{k}\right\}.\label{EMC}
\end{align}
\end{conj}

The two simple constructions showing that \eqref{EMC} is best possible (if true) are
$$\mathcal{E}(n,k,s)=\left\{E\in{[n]\choose k}\colon E\cap [s]\neq \emptyset\right\} \text{ and } {[(s+1)k-1]\choose k}.$$
The case $s=1$ is the classical Erd\H{o}s-Ko-Rado Theorem \cite{katona1964intersection}. Erd\H{o}s and Gallai \cite{erdHos1959maximal} proved \eqref{EMC} for $k=2$. The $k=3$ case is settled in \cite{Frankl2017On}. For general $k$,  Erd\H{o}s proved that \eqref{EMC} is true and up to isomorphic $\mathcal{E}(n,k,s)$ is the only optimal family provided that $n>n_0(k,s)$. The bounds for $n_0(k,s)$ were subsequently improved by Bollob\'{a}s, Daykin and Erd\H{o}s \cite{bollob1976sets}, Huang, Loh and Sudakov \cite{huang2012size}. The current best bounds establish \eqref{EMC} for $n>(2s+1)k$ (\hspace{-0.001em}\cite{frankl2013improved}) and for $s>s_0$, $n>\frac{5}{3}sk$ (\hspace{-0.001em}\cite{frankl2018erd}).

 A rainbow version of Erd\H{o}s matching conjecture was proposed by Aharoni, Howardand \cite{aharoni2017rainbow} and independently Huang, Loh, Sudakov \cite{huang2012size}.

\begin{conj}[\hspace{-0.001em}\cite{aharoni2017rainbow,huang2012size}]\label{conjrEMC}
  Let $\hf_1,\hf_2,\dots,\hf_s\subseteq\binom{[n]}{k}$. If
\begin{align}\label{conj-2}
|\hf_i|> \max\left\{\binom{ks-1}{k},\binom{n}{k}-\binom{n-s+1}{k}\right\},
\end{align}
for all $1\le i\le s$, then $\hf_1,\hf_2,\dots,\hf_s$ contain a  rainbow matching.
\end{conj}
\noindent Using a junta method, Keller and Lifshitz \cite{keller2017junta} showed the validity of \eqref{conj-2} for $n>f(s)k$, where $f(s)$ is an unspecified and a very quickly growing function of $s$. In \cite{frankl2019simple}, Frankl and Kupavskii improved the bound to $12sk(2+\log s)$. By using the absorbing method, Lu, Wang and Yu \cite{lu2020better} proved \eqref{conj-2} for $n\geq 2k(s+1)$ and $s>s_0(k)$. By applying the concentration inequality for random matchings developed in \cite{frankl2018erd}, Kupavskii\cite{kupavskii2021rainbow}  proved \eqref{conj-2} for $n\geq 3e(s+1)k$ and $s>10^7$.

Let $n_1,\dots,n_\ell,k_1,\dots,k_\ell$ be integers with $n_1/k_1\leq \dots \leq n_\ell/k_\ell$ and let $V_1,\dots,V_\ell$ be pairwise disjoint sets with $|V_i|=n_i$. Define the direct product $\sqcup_{i=1}^\ell \binom{V_i}{k_i}$ as the collection of all subsets $F$ of $\cup_{i=1}^\ell V_i$ with $|F\cap V_i|= k_i$ for $i=1,\dots,\ell$. By the algebraic method, Frankl\cite{frankl1996erdos} proved an Erd\H{o}s-Ko-Rado theorem for direct products as follows:

\begin{thm}[\hspace{-0.001em}\cite{frankl1996erdos}]\label{EKR-directprod}
Suppose that $\hf\subseteq \sqcup_{i=1}^\ell \binom{V_i}{k_i}$ is intersecting and $2\leq n_1/k_1\leq \dots \leq n_\ell/k_\ell$. Then
\[
|\hf| \leq \frac{k_1}{n_1}\prod_{i=1}^\ell\binom{n_i}{k_i}.
\]
\end{thm}

In this paper, we mainly consider a direct product version of Erd\H{o}s matching conjecture.  We determine the maximum number of edges in a subfamily of the direct product with matching number at most $s$ for each $n_i$ sufficiently large.

\begin{thm}\label{main-1}
Let $\ell,s,n_1,\dots,n_\ell,k_1,\dots,k_\ell$ be integers with $n_i\geq 4\ell^2k_i^2s$ for $i=1,\dots,\ell$. Let $\hf\subseteq \sqcup_{i=1}^\ell \binom{V_i}{k_i}$ with $|V_i|=n_i$. If $\nu(\hf)\leq s$, then
\begin{align}\label{eq-main-1}
|\hf| \leq \max_{1\leq i\leq \ell}\left[\binom{n_i}{k_i}-\binom{n_i-s}{k_i}\right]\prod_{j\neq i \atop 1\leq j\leq \ell}\binom{n_j}{k_j}.
\end{align}
\end{thm}
Let $S_i$ be an $s$-subset of $V_i$. Set $\vec{ n}=(n_1,n_2,\dots,n_\ell)$ and $\vec{k}=(k_1,k_2,\dots,k_\ell)$. Define
\[
\mathcal{E}_i(\vec{\mathstrut n},\vec{\mathstrut k},s)=\left\{E\in \bigsqcup_{j=1}^\ell \binom{V_j}{k_j}\colon E\cap S_i\neq \emptyset\right\}.
\]
The constructions $\mathcal{E}_i(\vec{\mathstrut n},\vec{\mathstrut k},s)$ show that \eqref{eq-main-1} is best possible.

By the random matching technique developed by Frankl and Kupavskii\cite{frankl2018erd}, we obtain a condition on the number of edges for a family of direct products to contain a rainbow matching for each $n_i$ sufficiently large.

\begin{thm}\label{main-2}
Let $\ell,s,n_1,\dots,n_\ell,k_1,\dots,k_\ell$ be integers with  $n_i\geq 8\ell^2k_i^2s$ for $i=1,\dots,\ell$. Let  $\hf_1,\hf_2,\dots,\hf_s\subseteq \sqcup_{i=1}^\ell \binom{V_i}{k_i}$ with $|V_i|=n_i$. If $\hf_1,\hf_2,\dots,\hf_s$ are rainbow matching free, then there exists $1\leq t\leq s$ such that
\begin{align}\label{eq-main-2}
  |\hf_t|\leq  \max_{1\leq i\leq \ell}\left[\binom{n_i}{k_i}-\binom{n_i-s+1}{k_i}\right]\prod_{j\neq i\atop 1\leq j\leq \ell}\binom{n_j}{k_j}.
\end{align}
The constructions $\hf_1=\hf_2=\dots=\hf_s=\mathcal{E}_i(\vec{\mathstrut n},\vec{\mathstrut k},s)$ show that \eqref{eq-main-2} is best possible.

\end{thm}

\section{Extremal problem for matchings on direct products}

Frankl and the first author \cite{frankl2021sum} proved an upper bound on the sum of sizes of overlapping families. In this section, by a probabilistic argument we transfer it to an upper bound on the size of  $\hf\subset \sqcup_{i=1}^\ell \binom{V_i}{k_i}$ with matching number at most $s$. By using this upper bound, we prove Theorem \ref{main-1}.

Let $\hf_1, \hf_2,\dots, \hf_m\subset \binom{[n]}{k}$ and $m\geq s+1$. If one cannot choose pairwise disjoint edges from $s+1$ distinct families of $\hf_1, \hf_2,\dots, \hf_m$, we say these families are $s$-overlapping.

 \begin{lem}[\hspace{-0.001em}\cite{frankl2021sum}]\label{lem1}
Let $\hf_1, \hf_2,\dots, \hf_m\subset \binom{[n]}{k}$ be $s$-overlapping families with $n\geq (s+1)k$ and $m\geq s+1$. Then
\[
\sum_{i=1}^m |\hf_i|\le \max\left\{
  s\binom{n}{k},
  ms\binom{n-1}{k-1} \right\}.
\]
\end{lem}

By Lemma \ref{lem1}, we prove the following lemma via a probabilistic argument.

\begin{lem}\label{lem2}
Let $\hf\subseteq \sqcup_{i=1}^\ell \binom{V_i}{k_i}$ with $|V_i|=n_i$ for $i=1,\dots,\ell$ and $s+1\leq n_1/k_1\leq \dots\leq n_\ell/k_\ell$. If $\nu(\hf)\leq s$, then
\[
|\hf|\le\frac{(s+1)k_1}{n_1} \prod_{i=1}^{\ell}\binom{n_i}{k_i}.
\]
\end{lem}
\proof  Let $m=\lfloor\frac{n_2}{k_2}\rfloor$ and let $B_1,B_2,\dots,B_m$ be an $m$-matching chosen from $\sqcup_{i=2}^\ell \binom{V_i}{k_i}$ uniformly at random. Define
 \[
\hf(B_i) =\left\{A\in \binom{V_1}{k_1} \colon A\cup B_i \in \hf\right\}
\]
for $i=1,\dots,m$. Since $\nu(\hf)\leq s$, $\hf(B_1),\dots,\hf(B_m)$ are $s$-overlapping. Note that $m\geq s+1$. By Lemma \ref{lem1}, we have
\[
\sum_{i=1}^m |\hf(B_i)| \leq \max\left\{s\binom{n_1}{k_1}, ms\binom{n_1-1}{k_1-1}\right\}.
\]
$n_1/k_1\leq n_2/k_2$ implies that $n_1/k_1\leq  \lfloor n_2/k_2\rfloor+1=m+1$. It follows that
\begin{align}\label{eq-1}
\sum_{i=1}^m |\hf(B_i)| \leq (m+1)s\binom{n_1-1}{k_1-1}.
\end{align}
Taking the expectation on both sides of $\eqref{eq-1}$, we obtain that
 \begin{align}\label{eq-2}
\sum_{i=1}^{m} \ex|\hf(B_i)| \leq (m+1)s\binom{n_1-1}{k_1-1}.
\end{align}
Since $B_1,B_2,\dots,B_m$ are chosen from $\sqcup_{i=2}^\ell \binom{V_i}{k_i}$ uniformly, we have
 \begin{align}\label{eq-3}
\ex|\hf(B_i)| = \sum_{B\in \bigsqcup_{i=2}^\ell \binom{V_i}{k_i}}  |\hf(B)|\cdot \frac{1}{\prod_{i=2}^\ell \binom{n_i}{k_i}}= \frac{|\hf|}{\prod_{i=2}^\ell\binom{n_i}{k_i}}.
 \end{align}
By \eqref{eq-2} and \eqref{eq-3},  we have
\[
|\hf| \leq \left(1+\frac{1}{m}\right)s \binom{n_1-1}{k_1-1} \prod_{i=2}^\ell\binom{n_i}{k_i}.
\]
Since $m=\lfloor\frac{n_2}{k_2}\rfloor >s$, it follows that
\[
\left(1+\frac{1}{m}\right)s < s+1
\]
and the lemma follows.
\qed

The following operation, called shifting, was invented by Erd\H{o}s, Ko and Rado \cite{ERD1961INTERSECTION} and further developed by Frankl \cite{frankl1987shifting}. Let $1\leq i< j\leq n$, $\hf\subset{[n]\choose k}$. Define
$$S_{ij}(\hf)=\{S_{ij}(F)\colon F\in\hf\},$$
where
$$S_{ij}(F)=\left\{
                \begin{array}{ll}
                  (F\setminus\{j\})\cup\{i\}, & j\in F, i\notin F \text{ and } (F\setminus\{j\})\cup\{i\}\notin \hf; \\[5pt]
                  F, & \hbox{otherwise.}
                \end{array}
              \right.
$$
It is well known (cf. \cite{frankl1987shifting}) that shifting does not increase the matching number.

Let us recall some properties of the shifting operator. Let $\hf\subseteq \sqcup_{i=1}^\ell \binom{V_i}{k_i}$ with $V_i=\{v_{i,1},v_{i,2},\dots,v_{i,n_i}\}$
for $i=1,\dots,\ell$. Define a partial order $\prec$ on $\cup_{i=1}^\ell V_i$ such that
\[
v_{i,1}\prec v_{i,2}\prec\dots\prec v_{i,n_i}
\]
for each $i$ and vertices from different parts  are incomparable. A family $\hf\subseteq \sqcup_{i=1}^\ell \binom{V_i}{k_i}$ is called shifted if $S_{ab}(\hf)=\hf$ for all $a,b\in \cup_{i=1}^\ell V_i$ with $a\prec b$. Frankl\cite{frankl1987shifting} proved that repeated application of shifting to any family  eventually produces a shifted family.

Let $k=k_1+k_2+\dots+k_\ell$. For two different edges  $A=\{a_1,a_2,\dots,a_k\}$ and $B=\{b_1,b_2,\dots,b_k\}$ in $\sqcup_{i=1}^\ell \binom{V_i}{k_i}$, define $A \prec B$ iff there exists a permutation $\sigma_1\sigma_2\dots\sigma_k$ of $[k]$ such that $a_j\prec b_{\sigma_j}$ or $a_j=b_{\sigma_j}$  for  $j=1,\dots,k$. It is shown in \cite{frankl1987shifting} that if $\hf$ is a shifted family, $A\prec B$ and $B\in \hf$ then $A\in \hf$ as well.

In the proof of the following lemma, we need the following inequality, which were already used in \cite{bollob1976sets}.  For $n\geq (2s+1)k$,
\begin{align}\label{ineq-6}
\binom{n-s}{k}/\binom{n}{k} \geq \left(1-\frac{s}{n-k}\right)^k\geq 1-\frac{ks}{n-k} \geq \frac{1}{2}.
\end{align}

\begin{lem}\label{lem3}
  Let $\ell,s,n_1,\dots,n_\ell,k_1,\dots,k_\ell$ be integers with  $n_i\geq 4\ell^2k_i^2s$ for $i=1,\dots,\ell$. If $\hf\subseteq \sqcup_{i=1}^\ell \binom{V_i}{k_i}$ has matching number at most $s$, then
  \[ |\hf|\le \prod_{i=1}^\ell\binom{n_i}{k_i}- \min_{x_1+\dots+x_\ell=s}\prod_{i=1}^\ell\binom{n_i-x_i}{k_i}
  \]
\end{lem}

\proof  Without loss of generality, we may assume that $n_1/k_1\leq \dots \leq n_\ell/k_\ell$ and $\hf$ is shifted. We prove the lemma by induction on $s$. The case $s=1$ is verified in Theorem \ref{EKR-directprod}. Assume that the lemma holds for the case $s-1$. Let $V_i=\{v_{i,j}\colon j\in [n_i]\}$. For each $x\in \cup_{i=1}^\ell V_i$, define
\[
\hf(x)=\{F\setminus \{x\} \colon x\in F\in \hf\}\mbox{ and } \hf(\overline{x})=\{F\colon x\notin F\in \hf\}.
\]
We distinguish two cases.

\textbf{Case 1.}  $\nu(\hf(\overline{v_{i,1}}))\le s-1$ for some $i\in [\ell]$.

By the induction hypothesis, we have
\[ |\hf(\overline{v_{i,1}})|\le \binom{n_i-1}{k_i}\prod_{j\neq i}\binom{n_j}{k_j}- \min_{y_1+\dots+y_\ell=s-1} \binom{n_i-1-y_i}{k_i}\prod_{j\neq i}\binom{n_j-y_j}{k_j}.\]
Then,
\begin{align*}
  |\hf|
  &=|\hf(\overline{v_{i,1}})|+|\hf({v_{i,1}})|\\[5pt]
  &\le \binom{n_i-1}{k_i}\prod_{j\neq i}\binom{n_j}{k_j}- \min_{y_1+\dots+y_\ell=s-1} \binom{n_i-1-y_i}{k_i}\prod_{j\neq i}\binom{n_j-y_j}{k_j}+\binom{n_i-1}{k_i-1}\prod_{j\neq i}\binom{n_j}{k_j}\\[5pt]
  &=\prod_{j=1}^\ell\binom{n_j}{k_j}- \min_{y_1+\dots+y_\ell=s-1} \binom{n_i-1-y_i}{k_i}\prod_{j\neq i}\binom{n_j-y_j}{k_j}\\[5pt]
  &\le\prod_{j=1}^\ell\binom{n_j}{k_j}-\min_{x_1+x_2+\dots+x_\ell=s}\prod_{j=1}^\ell\binom{n_j-x_j}{k_j}
\end{align*}
and the lemma holds.

\textbf{Case 2.}  $\nu(\hf(\overline{v_{i,1}}))=s$ for all $i\in [\ell]$.

Let $M$ be an $s$-matching in $\hf(\overline{v_{i,1}})$ and let $U$ be the set of vertices covered by $M$. Since $\nu(\hf)\leq s$, each edge in $\hf(v_{i,1})$ intersects $U$. It follows that
\begin{align*}
|\hf({v_{i,1}})| \le \prod_{j=1}^\ell\binom{n_j'}{k_j'}-\prod_{j=1}^\ell\binom{n_j'-sk_j}{k_j'},
\end{align*}
where $n_j'=n_j, k_j'=k_j$ for all $j\neq i$ and $n_i'=n_i-1, k_i'=k_i-1$. Then, we have
\begin{align*}
|\hf({v_{i,1}})| &\le \sum_{p=1}^\ell\left[\binom{n_p'}{k_p'}-\binom{n_p'-sk_p}{k_p'}\right]\prod_{j\leq p-1}\binom{n_j'-sk_j}{k_j'}\prod_{j\geq p+1}\binom{n_j'}{k_j'}\\[5pt]
\leq & \sum_{p=1}^\ell sk_p\binom{n_p'-1}{k_p'-1}\prod_{j\leq p-1}\binom{n_j'-sk_j}{k_j'}\prod_{j\geq p+1}\binom{n_j'}{k_j'}\\[5pt]
\leq & \sum_{p=1}^\ell sk_p\binom{n_p'-1}{k_p'-1}\prod_{j\leq p-1}\binom{n_j'}{k_j'}\prod_{j\geq p+1}\binom{n_j'}{k_j'}\\[5pt]
\leq & \sum_{p=1}^\ell  \frac{sk_pk_p'}{n_p'}\binom{n_p'}{k_p'}\prod_{j\leq p-1}\binom{n_j'}{k_j'}\prod_{j\geq p+1}\binom{n_j'}{k_j'}\\[5pt]
=&\left(\sum_{p=1}^\ell  \frac{sk_pk_p'}{n_p'}\right) \prod_{j=1}^\ell\binom{n_j'}{k_j'}.
\end{align*}
Note that $\frac{k_i'}{n_i'} = \frac{k_i-1}{n_i-1} \leq \frac{k_i}{n_i}$ and $\binom{n_i'}{k_i'} =\frac{k_i}{n_i}\binom{n_i}{k_i}$. It follows that
\begin{align}\label{ineq-1}
|\hf({v_{i,1}})| \le \frac{k_i}{n_i}\left(\sum_{p=1}^\ell  \frac{sk_p^2}{n_p}\right) \prod_{j=1}^\ell \binom{n_j}{k_j}.
\end{align}

Let $M'$ be an $s$-matching in $\hf$ and let $U'$ be the set of vertices covered by $M'$. Since $\hf$ is shifted, we may assume that
\[
U'=\bigcup_{i=1}^\ell\{v_{i,1},\dots,v_{i,sk_i}\}.
\]
For each $i\in [\ell]$ and  $j\in [sk_i]$, define
\[
\hf[v_{i,j}]= \left\{T\in \bigsqcup_{p=1}^\ell \binom{V_p'}{k_p'} \colon T\cup\{v_{i,j}\} \in \hf\right\},
\]
where $V_p'=V_p,\ k_p'=k_p$ for $p\neq i$ and $V_i'=V_i \setminus\{v_{i,1},\dots,v_{i,j}\},\  k_i'=k_i-1$. By shiftedness, we have
\begin{align}\label{ineq-2}
\hf[v_{i,1}]\supseteq \hf[v_{i,2}] \supseteq \dots\supseteq \hf[v_{i,sk_i}]
\end{align}
for $i=1,2,\dots,\ell$.

We claim that $\nu(\hf[v_{i,s+1}])\leq s$ for $i=1,2,\dots,\ell$. Otherwise, let $E_1,E_2,\dots,E_{s+1}$ be a matching of size $s+1$ in  $\hf[v_{i,s+1}]$. By \eqref{ineq-2},  $E_1\cup\{v_{i,1}\},E_2\cup\{v_{i,2}\},\dots,E_{s+1}\cup\{v_{i,s+1}\}$ is a matching of size $s+1$ in $\hf$, a contradiction. Thus $\nu(\hf[v_{i,s+1}])\leq s$ for $i=1,2,\dots,\ell$. Note that $n_i\geq (s+1)k_i$ implies that $\frac{n_i-s-1}{k_i-1} \geq \frac{n_i}{k_i}\geq \frac{n_1}{k_1}$. By Lemma \ref{lem2}, we get for each $i$
\begin{align}\label{ineq-3}
|\hf[v_{i,s+1}]| \leq \frac{(s+1)k_1}{n_1} \binom{n_i-s-1}{k_i-1}\prod_{j\neq i}\binom{n_j}{k_j} \leq \frac{(s+1)k_1k_i}{n_1n_i} \prod_{j=1}^\ell\binom{n_j}{k_j}.
\end{align}
Since $M'$ is a largest matching in $\hf$, $U'$ is a vertex cover of $\hf$. By \eqref{ineq-2}, it follows that
\begin{align}\label{eq-4}
|\hf| \leq \sum_{i=1}^\ell \sum_{j=1}^{sk_i} |\hf[v_{i,j}]|\leq \sum_{i=1}^\ell\left(s|\hf[v_{i,1}]|+(sk_i-s)|\hf[v_{i,s+1}]|\right).
\end{align}
Substitute \eqref{ineq-1} and \eqref{ineq-3} into \eqref{eq-4}, we arrive at
\begin{align*}
|\hf| &\leq \sum_{i=1}^\ell \left(s\frac{k_i}{n_i}\left(\sum_{p=1}^\ell \frac{sk_p^2}{n_p}\right) \prod_{j=1}^\ell \binom{n_j}{k_j}+s(k_i-1)\frac{(s+1)k_1k_i}{n_1n_i} \prod_{j=1}^\ell\binom{n_j}{k_j}\right)\\[5pt]
&\leq  s\prod_{j=1}^\ell\binom{n_j}{k_j}  \sum_{i=1}^\ell\left(\frac{k_i}{n_i}\cdot \sum_{p=1}^\ell \frac{sk_p^2}{n_p} +\frac{2sk_1k_i^2}{n_1n_i} \right).
\end{align*}
Since $n_i\geq 4\ell^2k_i^2s$, we have
\[
\sum_{p=1}^\ell \frac{sk_p^2}{n_p} \leq \frac{1}{4\ell},\mbox{ and } \frac{2sk_i^2}{n_i}\leq \frac{1}{2\ell^2}\leq \frac{1}{4\ell}.
\]
Moreover, $\frac{k_i}{n_i}\leq \frac{k_1}{n_1}$. Hence, by inequality \eqref{ineq-6} we obtain that
\[
|\hf| \leq \frac{s}{2}\binom{n_1-1}{k_1-1}\prod_{j=2}^\ell\binom{n_j}{k_j}\leq  s\binom{n_1-s}{k_1-1}\prod_{j=2}^\ell\binom{n_j}{k_j}\leq \left(\binom{n_1}{k_1}-\binom{n_1-s}{k_1}\right)\prod_{j=2}^\ell\binom{n_j}{k_j}
\]
and this completes the proof.
\qed

\begin{proof}[Proof of Theorem \ref{main-1}.] By Lemma \ref{lem3}, we are left to show that
\[
 \min_{x_1+\dots+x_\ell=s}\prod_{i=1}^\ell\binom{n_i-x_i}{k_i} =\min_{1\leq i\leq \ell}\binom{n_i-s}{k_i}\prod_{j\neq i}\binom{n_j}{k_j}.
\]
Set
\[
f(x) =\binom{n_i-x}{k_i}\binom{n_j-c+x}{k_j}.
\]
Computing its derivative, we have
\[
f'(x) = \binom{n_i-x}{k_i}\binom{n_j-c+x}{k_j} \left[\sum_{q=0}^{k_j-1} \frac{1}{n_j-c+x-q}-\sum_{p=0}^{k_i-1} \frac{1}{n_i-x-p}\right].
\]
Set
\[
g(x) = \sum_{q=0}^{k_j-1} \frac{1}{n_j-c+x-q}-\sum_{p=0}^{k_i-1} \frac{1}{n_i-x-p}.
\]
It is easy to see that $g(x)$ is a decreasing function in the range $[k_j-n_j+c, n_i-k_i]$. Moreover,
\[
g(k_j-n_j+c)=\sum_{q=0}^{k_j-1} \frac{1}{k_j-q}-\sum_{p=0}^{k_i-1} \frac{1}{n_i-p+n_j-k_j-c}> 1- \frac{k_i}{n_i-k_i}>0
\]
and
\[
g(n_i-k_i) = \sum_{q=0}^{k_j-1} \frac{1}{n_j-c+n_i-k_i-q}-\sum_{p=0}^{k_i-1} \frac{1}{k_i-p} <\frac{k_j}{n_j-k_j}-1 <0.
\]
Let $x_0$ be the unique zero of $g(x)$ in the range $[k_j-n_j+c, n_i-k_i]$. It follows that $f(x)$ is increasing in the range $[k_j-n_j,x_0]$ and decreasing in the range $[x_0,n_i-k_i]$. Then we have
\[
\min_{0\leq x\leq c}f(x) =\min\left\{\binom{n_i-c}{k_i}\binom{n_j}{k_j}, \binom{n_i}{k_i}\binom{n_j-c}{k_j}\right\}.
\]
Thus,
\[
 \min_{x_1+\dots+x_\ell=s}\prod_{i=1}^\ell\binom{n_i-x_i}{k_i} =\min_{1\leq i\leq \ell} \binom{n_i-s}{k_i}\prod_{j\neq i}\binom{n_j}{k_j}
\]
and the theorem follows.
\end{proof}

\section{Extremal problem for rainbow matchings on direct products}

Frankl and Kupavskii \cite{frankl2018erd} proved a concentration inequality for the intersections of a $k$-uniform hypergraph and a random matching. In this section, by the concentration property of the bipartite graph defined on two random matchings, we establish a condition on the number of edges for a family of direct products to contain a rainbow matching.

Aharoni and Howard \cite{aharoni2017rainbow} provided a tight condition on the number of edges for a family of bipartite graphs to contain a rainbow matching.

\begin{lem}[\hspace{-0.001em}\cite{aharoni2017rainbow}]\label{lem:aharoni-howard}
Let $G_1,G_2,\dots,G_s$ be bipartite graphs with the same partite sets $L$ and $R$ such that $|L|=|R|=m\geq s$. If $e(G_i)>(s-1)m$ for every $i\in [s]$, then $G_1,G_2,\dots,G_s$ contain a rainbow matching.
\end{lem}

We recall that the Kneser graph $KG(n,k)$ is the graph on the vertex set $\binom{[n]}{k}$
and with the edge set formed by pairs of disjoint sets. Lovasz \cite{lovasz1979shannon} determined all the eigenvalues of the adjacent matrix of $KG(n,k)$.

\begin{lem}[\hspace{-0.001em}\cite{lovasz1979shannon}]\label{lem:lovasz1979shannon}
The eigenvalues of the adjacent matrix of $KG(n,k)$ are $(-1)^i\binom{n-k-1}{k-i}$ with multiplicity $\binom{n}{i}-\binom{n}{i-1}$, $i=0,1,\dots,k$. {\textup(}Here $\binom{n}{-1}=0$.{\textup)}
\end{lem}

We also need the following result due to Alon and Chung \cite{alon1988explicit}.

\begin{lem}[\hspace{-0.001em}\cite{alon1988explicit}]\label{lem:alon-chung}
  Let $G$ be a $d$-regular graph on $n$ vertices and $A$ be its  adjacent matrix. Let  $d=\lambda_1\geq\lambda_2\geq\dots\geq \lambda_n\geq-d$ be the eigenvalue of $A$ and let $\lambda=\max\{|\lambda_2|,|\lambda_n|\}$. For $S\subseteq V(G)$ with $|S|=\alpha n$, we have
\[ \left| e(G[S])-\frac dn\cdot\frac{(\alpha n)^2}{2}   \right|\leq \frac 12 \lambda\alpha(1-\alpha)n.
\]
\end{lem}

By considering the bipartite graph defined on two random matchings, we prove the following lemma, which is key to the proof of Theorem \ref{main-2}.

\begin{lem}\label{lem:main}
Let $\ell,s,n_1,\dots,n_\ell,k_1,\dots,k_\ell$ be integers with $3s\leq n_1/k_1\leq \dots \leq n_\ell/k_\ell$. Let  $\hf_1,\hf_2,\dots,\hf_s\subseteq \sqcup_{i=1}^\ell \binom{V_i}{k_i}$ with $|V_i|=n_i$ for $i=1,\dots,\ell$. If $\hf_1,\hf_2,\dots,\hf_s$ are rainbow matching free, then there exists $1\leq t\leq s$ such that
\[
  |\hf_t|<\frac{6sk_1}{n_1}\prod_{1\leq j\leq \ell}\binom{n_j}{k_j}.
\]
\end{lem}
\proof Suppose that
\[ |\hf_t|= \frac{6sk_1}{n_1}\prod_{1\leq j\leq \ell}\binom{n_j}{k_j}\]
for all $1\leq t\leq s$. Set $m=\lfloor n_1/k_1\rfloor$. Let $\ha=\{A_{1},\dots,A_{m}\}$ be an $m$-matching chosen from $\binom{V_1}{k_1}$ uniformly at random and $\hb=\{B_{1},\dots,B_{m}\}$ be an $m$-matching chosen from $\sqcup_{i=2}^\ell \binom{V_i}{k_i}$ uniformly at random. For each $t=1,\dots,s$, we construct a bipartite graph $G_t$ with partite sets $\ha$, $\hb$ where we have an edge $(A_i,B_j)$ iff $A_i\cup B_j \in \hf_t$. Note that a rainbow matching in $G_1,G_2,\dots,G_s$ gives a rainbow matching in $\hf_1,\hf_2,\dots,\hf_s$. Thus, by Lemma \ref{lem:aharoni-howard} and the union bound, it suffices to show that
\begin{align*}
  \Pr(e(G_t)\leq (s-1)m)<\frac 1s
\end{align*}
for each $1\leq t\leq s$. Let $X=e(G_t)$ and $X_{ij}$ be the indicator function of the event that $A_i\cup B_j\in\hf_t$. Clearly, we have
\[
X=\sum_{1\leq i,j\leq m} X_{ij}.
\]
Set $\alpha=|\hf_t|/\prod_{j=1}^\ell\binom{n_j}{k_j}$. Take expectation on both sides, we obtain that
\[
\ex(X)  =  \sum_{1\leq i,j\leq m} \Pr[A_i\cup B_j\in\hf_t] = \sum_{1\leq i,j\leq m} \frac{|\hf_t|}{\prod_{j=1}^\ell\binom{n_j}{k_j}} = \alpha m^2.
\]
Since $n_1/k_1\geq 3s$, we have
\begin{align}\label{ineq-4}
\alpha m =\frac{6sk_1}{n_1} \cdot \left\lfloor \frac{n_1}{k_1}\right\rfloor \geq \frac{6sk_1}{n_1} \cdot  \frac{n_1-k_1}{k_1} \geq 6s\left(1-\frac{k_1}{n_1}\right)\geq 6s-2.
\end{align}
By the Chebyshev inequality, it follows that
\begin{align}\label{eq:chebyshev}
  \Pr(X\leq (s-1)m)\leq \Pr\left(X\leq \frac{\ex (X)}{6}\right) \leq \Pr\left(|X-\ex(X)|\geq \frac{5}{6}\ex(X)\right)\leq  \frac{36}{25} \cdot \frac{\Var(X)}{\alpha^2m^4}.
\end{align}
Now we need to give an upper bound for $\Var(X)$. For  $1\leq i,i',j,j'\leq m$, define
\[
\theta(i,j,i',j') = \Pr\left(A_i\cup B_j, A_{i'}\cup B_{j'}\in \hf_t\right)-\alpha^2.
\]
Then,
\begin{align}\label{eq:3.3}
\Var(X)&=\ex(X^2)-\ex(X)^2\nonumber\\[5pt]
& =  \sum_{i,j,i',j'\in [m]} \left(\ex [X_{ij}X_{i'j'}] - \ex [X_{ij}]\ex[X_{i'j'}]\right)\nonumber\\[5pt]
&= \sum_{i,j,i',j'\in [m]}  \left(\Pr\left(A_i\cup B_j, A_{i'}\cup B_{j'}\in \hf_t\right)-\Pr\left(A_i\cup B_j\in \hf_t\right)\Pr\left( A_{i'}\cup B_{j'}\in \hf_t\right)\right)\nonumber\\[5pt]
&= \sum_{i,j,i',j'\in [m]} \theta(i,j,i',j').
\end{align}
We derive upper bounds on $\theta(i,j,i',j')$ for $i,j,i',j'\in [m]$ by distinguishing four cases.

{\bf Case 1.} $i\neq i'$ and $j\neq j'$.

Let $H$ be a graph on the vertex set $\sqcup_{i=1}^\ell \binom{V_i}{k_i}$
and with the edge set formed by pairs of disjoint sets in $\sqcup_{i=1}^\ell \binom{V_i}{k_i}$. Let $\lambda$ be the second largest absolute eigenvalue of adjacent matrix of $H$. Frankl \cite{frankl1996erdos} observed that the adjacency matrix of $H$ is the Kronecker product of the adjacency matrices of the Kneser graphs. By Lemma \ref{lem:lovasz1979shannon}, it follows that
\[
\lambda=\frac{k_1}{n_1-k_1}\prod_{1\leq p\leq \ell}\binom{n_p-k_p}{k_p}.
\]
Since $i\neq i'$ and $j\neq j'$, we have $(A_i\cup B_j)\cap(A_{i'}\cup B_{j'})=\emptyset$. Note that $A_{1},\dots,A_{m}$ and $B_{1},\dots,B_{m}$ are chosen uniformly at random. The probability that $A_i\cup B_j, A_{i'}\cup B_{j'}\in\hf_t$ equals to the probability that a uniform  chosen edge from $E(H)$ is an edge of the induce subgraph $H[\hf_t]$, that is,
\begin{align*}
  \Pr\Big(A_i\cup B_j, A_{i'}\cup B_{j'}\in \hf_t\Big)=\frac{e(H[\hf_t])}{e(H)}.
\end{align*}
Let $D$ be the degree of $H$ and $N$ be the number of vertices of $H$. It is easy to see that $D=\prod_{1\leq p\leq \ell}\binom{n_p-k_p}{k_p}$ and $N=\prod_{1\leq p\leq \ell}\binom{n_p}{k_p}$. Since $H$ is $D$-regular, we have $e(H) = \frac{DN}{2}$.   By Lemma \ref{lem:alon-chung}, we have
\begin{align}
\theta(i,j,i',j')  \leq \left|\frac{e(H[\hf_t])}{e(H)}-\alpha^2\right|\leq \frac{\lambda}{D}\alpha(1-\alpha)\leq \frac{k_1}{n_1-k_1}\alpha.\label{eq:lemMain-1}
\end{align}

{\bf Case 2.} $i\neq i'$ and $j=j'$.

For $B\in \sqcup_{i=2}^\ell \binom{V_i}{k_i}$, define
$\hf_t(B)=\left\{K\in \binom{V_1}{k_1}\colon K\cup B\in \hf_t\right\}$. Let $H$ be the Kneser graph on the vertex set $\binom{V_1}{k_1}$.
Since $i\neq i'$ and $j=j'$, the event that $A_i\cup B_j, A_{i'}\cup B_{j'}\in \hf_t$ is equivalent to the event that a uniform chosen edge of $H$ is an edge in $H[\hf_t(B_j)]$. Thus, \begin{align}\label{eq:24-3}
  \Pr\left(A_i\cup B_j, A_{i'}\cup B_{j'}\in \hf_t\right)&=\sum_{B\in\bigsqcup_{i=2}^\ell \binom{V_i}{k_i}}\Pr\left(A_i,A_i'\in \hf_t(B)|B_j=B\right)\Pr\left(B_j=B\right)\nonumber\\[5pt]
  &=\sum_{B\in\bigsqcup_{i=2}^\ell \binom{V_i}{k_i}}\frac{e(H[\hf_t(B)])}{e(H)}\cdot\Pr\left(B_j=B\right).
\end{align}
Set $\alpha(B)=|\hf_t(B)|/|\binom{n_1}{k_1}|$.
Let $D$ be the degree of $H$ and $N$ be the number of vertices in $H$. Clearly, $D=\binom{n_1-k_1}{k_1} \mbox{ and } N=\binom{n_1}{k_1}$.
Let $\lambda$ be the second largest absolute eigenvalue of adjacent matrix of $H$. By Lemma \ref{lem:lovasz1979shannon}, we have $\lambda=\binom{n_1-k_1-1}{k_1-1}$. By lemma \ref{lem:alon-chung}, it follows that
\begin{align}\label{eq:27-1}
  \left|\frac{e(H[\hf_t(B)])}{e(H)}-\alpha^2(B)\right|\leq \frac{\lambda}{D}\alpha(B)(1-\alpha(B)) \leq \alpha(B)(1-\alpha(B)).
\end{align}
By \eqref{eq:24-3} and \eqref{eq:27-1}, we obtain that
\begin{align}
\theta(i,j,i',j')
  & = \Pr\left(A_i\cup B_j, A_{i'}\cup B_{j'}\in \hf_t\right)-\alpha^2\nonumber\\[5pt]
  &\leq \frac{1}{\prod_{2\leq p\leq \ell}\binom{n_p}{k_p}}\sum_{B\in\bigsqcup_{i=2}^\ell \binom{V_i}{k_i}}\left(\alpha(B)(1-\alpha(B))+\alpha^2(B)\right) -\alpha^2\nonumber\\[5pt]
  &\leq \frac{1}{\prod_{2\leq p\leq \ell}\binom{n_p}{k_p}}\sum_{B\in\bigsqcup_{i=2}^\ell \binom{V_i}{k_i}}\alpha(B)\nonumber\\[5pt]
  &\leq \frac{1}{\prod_{2\leq p\leq \ell}\binom{n_p}{k_p}}\sum_{B\in\bigsqcup_{i=2}^\ell \binom{V_i}{k_i}}\frac{|\hf_t(B)|}{\binom{n_1}{k_1}}\nonumber\\[5pt]
   &\leq \frac{1}{\prod_{1\leq p\leq \ell}\binom{n_p}{k_p}}\sum_{B\in\bigsqcup_{i=2}^\ell \binom{V_i}{k_i}}|\hf_t(B)|\nonumber\\[5pt]
      &\leq \frac{|\hf_t|}{\prod_{1\leq p\leq \ell}\binom{n_p}{k_p}}\nonumber\\[5pt]
  &= \alpha\label{eq:lemMain-2}
\end{align}

{\bf Case 3.} $i=i'$ and $j\neq j'$.

Similarly, let $H$ be a graph on the vertex set $\sqcup_{i=2}^\ell \binom{V_i}{k_i}$
and with the edge set formed by pairs of disjoint sets in $\sqcup_{i=2}^\ell \binom{V_i}{k_i}$.  Let $\lambda$ be the second largest absolute eigenvalue of adjacent matrix of $H$. Since the adjacency matrix of $H$ is the Kronecker product of the adjacency matrices of the Kneser graphs. By Lemma \ref{lem:lovasz1979shannon}, it follows that
\[
\lambda=\frac{k_2}{n_2-k_2}\prod_{2\leq p\leq \ell}\binom{n_p-k_p}{k_p}.
\]
For each $A\in \binom{V_1}{k_1}$, define $\hf_t(A)=\left\{K\in \sqcup_{i=2}^\ell \binom{V_i}{k_i}\colon A\cup K\in \hf_t\right\}$.
Set $\alpha(A)=|\hf_t(A)|/ \prod_{2\leq p\leq \ell}\binom{n_p}{k_p}$.
Let $D$ be the degree of $H$ and let $N$ be the number of vertices in $H$. Clearly,
$D=\prod_{2\leq p\leq \ell}\binom{n_p-k_p}{k_p}$ and $N=\prod_{2\leq p\leq \ell}\binom{n_p}{k_p}$.
Then,
\begin{align*}
  \Pr\left(A_i\cup B_j, A_{i'}\cup B_{j'}\in \hf_t\right)=\sum_{A\in\binom{V_1}{k_1}}\frac{e(H[\hf_t(A)])}{e(H)}\cdot \Pr\left(A_i=A\right).
\end{align*}
By Lemma \ref{lem:alon-chung}, we have for $A\in \binom{V_1}{k_1}$
\[
\left|\frac{e(H[\hf_t(A)])}{e(H)}-\alpha^2(A)\right|\leq \frac{\lambda}{D}\alpha(A)(1-\alpha(A)) \leq \alpha(A)(1-\alpha(A)).
\]
Thus,
\begin{align}
 \theta(i,j,i',j')
  &\leq\Pr\left(A_i\cup B_j, A_{i'}\cup B_{j'}\in \hf_t\right)-\alpha^2\nonumber\\
  &\leq \frac{1}{\binom{n_1}{k_1}}\sum_{A\in\binom{V_1}{k_1}}\left(\alpha(A)(1-\alpha(A))+\alpha^2(A)\right)\nonumber\\
  &\leq  \frac{1}{\binom{n_1}{k_1}}\cdot\sum_{A\in\binom{V_1}{k_1}}\alpha(A)\nonumber\\
  &\leq \frac{1}{\binom{n_1}{k_1}}\sum_{A\in\binom{V_1}{k_1}}\frac{|\hf_t(A)|}{\prod_{2\leq p\leq \ell}\binom{n_p}{k_p}}\nonumber\\
  &= \alpha.\label{eq:lemMain-4}
\end{align}

{\bf Case 4.} $i= i'$ and $j=j'$.

In this case,  $\Pr\left(A_i\cup B_j,A_{i'}\cup B_{j'} \in \hf_t\right)=\Pr\left(A_i\cup B_j \in \hf_t\right)=\alpha$. Hence,
\begin{align}
  \theta(i,j,i',j')=\alpha-\alpha^2\leq \alpha.\label{eq:lemMain-3}
\end{align}

Substitute \eqref{eq:lemMain-1}, \eqref{eq:lemMain-2}, \eqref{eq:lemMain-4} and \eqref{eq:lemMain-3} into \eqref{eq:3.3}, we have
\[
  \Var(X)= \sum_{i,j,i',j'\in [m]} \theta(i,j,i',j')\leq \frac{k_1}{n_1-k_1}\alpha m^{2}(m-1)^2+2\alpha m^{3}+\alpha m^2.
\]
Note that $m=\lfloor\frac{n_1}{k_1}\rfloor\leq \frac{n_1}{k_1}$ implies $ \frac{k_1}{n_1-k_1} \leq \frac{1}{m-1}$. It follows that
\begin{align}
   \Var(X)\leq \frac{1}{m-1}\alpha m^2(m-1)^2+2\alpha m^{3}+\alpha m^2 = 3\alpha m^3.
  \label{eq:var}
\end{align}
Substitute  \eqref{eq:var} into \eqref{eq:chebyshev}, we arrive at
\[
\Pr(X\leq (s-1)m)\leq \frac{36}{25}\cdot\frac{3\alpha m^3}{\alpha^2m^{4}}= \frac{36\cdot 3}{25\alpha m}.
\]
Moreover, \eqref{ineq-4} implies that $\alpha m\geq 6s-2\geq 5s$. Thus, we obtain that
\[
\Pr(X\leq (s-1)m)\leq  \frac{36\times 3}{25\alpha m}\leq \frac{36\times 3}{25\times 5} \cdot \frac{1}{s}<\frac{1}{s},
\]
and the lemma follows.\qed

The following lemma shows that the shifting operator preserves the rainbow matching free property of a family of uniform hypergraphs, which is due to Huang, Loh and Sudakov \cite{huang2012size}.

\begin{lem}[Huang-Loh-Sudakov \cite{huang2012size}]\label{shifted-rainbow}
If the families $\hf_1, \dots , \hf_s\subseteq \binom{[n]}{k}$ are rainbow matching free and $i,j\in [n]$ with $i< j$, then $S_{ij}(\hf_1), . . . , S_{ij}(\hf_s)$ are still rainbow matching free.
\end{lem}

\begin{lem}\label{lem:main-2}
Let $\ell,s,n_1,\dots,n_\ell,k_1,\dots,k_\ell$ be integers with  $n_i\geq 8\ell^2k_i^2s$ for $i=1,\dots,\ell$. Let  $\hf_1,\hf_2,\dots,\hf_s\subseteq \sqcup_{i=1}^\ell \binom{V_i}{k_i}$ with $|V_i|=n_i$. If $\hf_1,\hf_2,\dots,\hf_s$ are rainbow matching free, then there exists $1\leq t\leq s$ such that
\[
  |\hf_t|\leq  \prod_{i=1}^\ell\binom{n_i}{k_i}- \min_{x_1+\dots+x_\ell=s-1}\prod_{i=1}^\ell\binom{n_i-x_i}{k_i}.
\]
\end{lem}
\begin{proof}
 Without loss of generality, we may assume that $n_1/k_1\leq n_2/k_2\leq \dots\leq n_\ell/k_\ell$. By Lemma \ref{shifted-rainbow}, we may assume that $\hf_i$ is shifted for all $i=1,\dots,s$. We prove the theorem by induction on $s$. The case $s=1$ is trivial. Assume that the lemma holds for the case $s-1$ and we have to show the lemma holds for the case $s$.

 Let $V_i=\{v_{i,j} \colon j\in [n_i]\}$. For every $x\in \cup_{i=1}^\ell V_i$ and $1\leq t\leq s$, define
 \[
 \hf_t(x)=\{F\setminus \{x\} \colon x\in F\in \hf_t\} \mbox{ and } \hf_t(\overline{x})=\{F \colon x\notin F\in \hf_t\}.
 \]
 We distinguish into two cases.

{\bf Case 1.} There exist $r\in [s]$ and $i\in [\ell]$ such that $\hf_1(\overline{v_{i,1}}),\dots,\hf_{r-1}(\overline{v_{i,1}}),\hf_{r+1}(\overline{v_{i,1}}),\dots,$\\
$\hf_s(\overline{v_{i,1}})$ are rainbow matching free.

By the induction hypothesis, there exists $t\in [s]\setminus\{r\}$ such that
\[ |\hf_{t}(\overline{v_{i,1}})|\le \binom{n_i-1}{k_i}\prod_{j\neq i}\binom{n_j}{k_j}- \min_{y_1+\dots+y_\ell=s-1} \binom{n_i-1-y_i}{k_i}\prod_{j\neq i}\binom{n_j-y_j}{k_j}.\]
Then,
\begin{align*}
  |\hf_{t}|
  &=|\hf_{t}(\overline{v_{i,1}})|+|\hf_{t}({v_{i,1}})|\\[5pt]
  &\le \binom{n_i-1}{k_i}\prod_{j\neq i}\binom{n_j}{k_j}- \min_{y_1+\dots+y_\ell=s-1} \binom{n_i-1-y_i}{k_i}\prod_{j\neq i}\binom{n_j-y_j}{k_j}+\binom{n_i-1}{k_i-1}\prod_{j\neq i}\binom{n_j}{k_j}\\[5pt]
  &=\prod_{j=1}^\ell\binom{n_j}{k_j}- \min_{y_1+\dots+y_\ell=s-1} \binom{n_i-1-y_i}{k_i}\prod_{j\neq i}\binom{n_j-y_j}{k_j}\\[5pt]
  &\le\prod_{j=1}^\ell\binom{n_j}{k_j}-\min_{x_1+x_2+\dots+x_\ell=s}\prod_{j=1}^\ell\binom{n_j-x_j}{k_j}
\end{align*}
and the lemma holds.

{\bf Case 2.} $\hf_1(\overline{v_{i,1}}),\dots,\hf_{r-1}(\overline{v_{i,1}}),\hf_{r+1}(\overline{v_{i,1}}),\dots,\hf_s(\overline{v_{i,1}})$ contain a rainbow matching for all $r\in [s]$ and  $i\in [\ell]$.

Let  $M$ be a rainbow matching in $\left\{\hf_1(\overline{v_{i,1}}),\dots,\hf_{r-1}(\overline{v_{i,1}}),\hf_{r+1}(\overline{v_{i,1}}),\dots,\hf_s(\overline{v_{i,1}})\right\}$  and let $U$ be the set of vertices covered by $M$. Since $\hf_1,\hf_2,\dots,\hf_s$ are rainbow matching free, each edge of $\hf_r(v_{i,1})$ intersects $U$. It follows that
\begin{align*}
|\hf_r({v_{i,1}})| \le \prod_{j=1}^\ell\binom{n_j'}{k_j'}-\prod_{j=1}^\ell\binom{n_j'-(s-1)k_j}{k_j'},
\end{align*}
where $n_j'=n_j, k_j'=k_j$ for all $j\neq i$ and $n_i'=n_i-1, k_i'=k_i-1$. Then,
\begin{align*}
|\hf_r({v_{i,1}})| &\le \sum_{p=1}^\ell\left[\binom{n_p'}{k_p'}-\binom{n_p'-(s-1)k_p}{k_p'}\right]\prod_{j\leq p-1}\binom{n_j'-(s-1)k_j}{k_j'}\prod_{j\geq p+1}\binom{n_j'}{k_j'}\\[5pt]
\leq & \sum_{p=1}^\ell(s-1)k_p\binom{n_p'-1}{k_p'-1}\prod_{j\leq p-1}\binom{n_j'-(s-1)k_j}{k_j'}\prod_{j\geq p+1}\binom{n_j'}{k_j'}\\[5pt]
\leq & \sum_{p=1}^\ell(s-1)k_p\binom{n_p'-1}{k_p'-1}\prod_{j\leq p-1}\binom{n_j'}{k_j'}\prod_{j\geq p+1}\binom{n_j'}{k_j'}\\[5pt]
\leq & \sum_{p=1}^\ell \frac{(s-1)k_pk_p'}{n_p'}\binom{n_p'}{k_p'}\prod_{j\leq p-1}\binom{n_j'}{k_j'}\prod_{j\geq p+1}\binom{n_j'}{k_j'}\\[5pt]
=&\left(\sum_{p=1}^\ell \frac{(s-1)k_pk_p'}{n_p'}\right) \prod_{j=1}^\ell\binom{n_j'}{k_j'}.
\end{align*}
Note that $\frac{k_i'}{n_i'} = \frac{k_i-1}{n_i-1} \leq \frac{k_i}{n_i}$ and $\binom{n_i'}{k_i'} =\frac{k_i}{n_i}\binom{n_i}{k_i}$. It follows that for every $i\in[\ell]$ and $r\in [s]$
\begin{align}\label{eq:27-3-14}
|\hf_r({v_{i,1}})| \le \frac{(s-1) k_i}{n_i}\left(\sum_{p=1}^\ell \frac{k_p^2}{n_p}\right) \prod_{j=1}^\ell\binom{n_j}{k_j}.
\end{align}

For $j\in[\ell]$,  let $S_j=\{v_{j,1},\dots,v_{j,s}\}$ and let $V_j' =V_j\setminus S_j$. For every $r\in [s]$ and $i\in[\ell]$, define
\[
\hf_r[v_{i,s}]= \left\{E\in  \bigsqcup_{j=1}^\ell\binom{V_j'}{k_j'}  \colon E\cup\{v_{i,s}\} \in \hf_r\right\},
\]
where $k_j'=k_j$ for all $j\neq i$ and $k_i'=k_i-1$.
Define
 $$\hf_r^*[v_{i,s}]=\left\{E\cup \{u\} \colon E\in \hf_r[v_{i,s}],\ u\in V_i'\setminus E\right\}.$$

{\bf Claim 1.} There exists $t\in [s]$ such that for all $i\in [\ell]$
\begin{align}\label{eq:26-1}
  |\hf_t^*[v_{i,s}]|\leq \frac{6sk_1}{n_1}\prod_{1\leq j\leq \ell}\binom{n_j-s}{k_j}.
\end{align}

\begin{proof}
Suppose to the contrary that for every $t\in [s]$ there is an $i_t \in [\ell]$ such that
\[
|\hf_t^*[v_{i_t,s}]|> \frac{6sk_1}{n_1}\prod_{1\leq j\leq \ell}\binom{n_j-s}{k_j}.
\]
Note that
\[
\hf_t^*[v_{i_t,s}] \subseteq \bigsqcup_{j=1}^\ell \binom{V_j'}{k_j}.
\]
By Lemma \ref{lem:main}, we conclude that $\hf_1^*[v_{i_1,s}],\dots,\hf_s^*[v_{i_s,s}]$ contain a rainbow matching. Let $F_1\in  \hf_1^*[v_{i_1,s}],\dots,F_s \in\hf_s^*[v_{i_s,s}]$ be such a matching. By the definition of $\hf_t^*[v_{i_t,s}]$, we see that there exist $E_1\subset F_1,\dots,E_s\subset F_s$ such that $E_1\cup \{v_{i_1,s}\}\in \hf_1, E_2\cup \{v_{i_2,s}\}\in \hf_2 \dots, E_s\cup \{v_{i_s,s}\} \in \hf_s$. Note that $\hf_1,\dots,\hf_s$ are all shifted. It follows that $E_1\cup \{v_{i_1,1}\}\in \hf_1, E_2\cup \{v_{i_2,2}\}\in \hf_2, \dots, E_s\cup \{v_{i_s,s}\} \in \hf_s$ is a rainbow matching, which contradicts the fact that $\hf_1,\dots,\hf_s$ are rainbow matching free.
\end{proof}

By Claim 1, there exists $t\in [s]$ such that \eqref{eq:26-1} holds for all $i\in [\ell]$. By the definition of $\hf_t^*[v_{i,s}]$, we see that for each $E\in \hf_t[v_{i,s}]$ there are  $|V_{i}'\setminus E|=(n_i-s)-(k_i-1)$ choices of $u$ such that $E\cup \{u\}\in \hf_t^*[v_{i,s}]$. Moreover, for each $F\in \hf_t^*[v_{i,s}]$ we have $|F\cap V_i'|=k_i$. It follows that at most $k_i$ sets in $\hf_t[v_{i,s}]$ are contained in $F$. Thus, we have
\begin{align}\label{eq:27-3-15}
  |\hf_t^*[v_{i,s}]|\geq \frac{n_i-s-k_i+1}{k_i}\cdot |\hf_t[v_{i,s}]|.
\end{align}
By \eqref{eq:26-1} and \eqref{eq:27-3-15}, we obtain that for all $i=1,2,\ldots,\ell$,
\begin{align}\label{eq:27-3-17}
  |\hf_t[v_{i,s}]|\leq \frac{6sk_1k_i}{n_1(n_i-s-k_i+1)}\prod_{1\leq j\leq \ell}\binom{n_j}{k_j}.
\end{align}
Note that $\hf_1,\dots,\hf_{t-1},\hf_{t+1},\dots,\hf_s$ contain a rainbow matching. Let $M'$ be such a matching and let $U'$ be the set of vertices covered by $M'$. By shiftedness, we may assume that
\[
U'=\bigcup_{i=1}^\ell\left\{v_{i,1},\dots,v_{i,(s-1)k_i}\right\}.
\]
Since $\hf_1,\dots,\hf_s$ are rainbow matching free, it follows that every edge of $\hf_t$ intersects $U'$. Thus,
\[
\hf_t = \sum_{i=1}^\ell \sum_{j=1}^{(s-1)k_i}\hf_t[v_{i,j}].
\]
By shiftedness, we have
\begin{align*}
\hf_t[v_{i,1}]\supseteq \hf_t[v_{i,2}] \supseteq \dots\supseteq \hf_t[v_{i,(s-1)k_i}]
\end{align*}
for $i=1,2,\dots,\ell$. Then
\begin{align}\label{ineq-5}
|\hf_t| \leq \sum_{i=1}^\ell(s-1)|\hf_t[{v_{i,1}}]|+\sum_{i=1}^\ell(k_i-1)(s-1)|\hf_t[v_{i,s}]|.
\end{align}
Substitute \eqref{eq:27-3-14} and \eqref{eq:27-3-17} into \eqref{ineq-5}, we arrive at
\begin{align*}
  |\hf_t|\leq \sum_{i=1}^\ell\Bigg((s-1)^2\Big(\sum_{p=1}^\ell\frac{k_p^2}{n_p}\Big)\frac{k_i}{n_i}\prod_{j=1}^\ell
  \binom{n_j}{k_j}+6s(s-1)\frac{k_i(k_i-1)}{n_i-s-k_i+1}\frac{k_1}{n_1}\prod_{j=1}^\ell\binom{n_j}{k_j}\Bigg).
  \end{align*}
 Since $\frac{k_i}{n_i}\leq \frac{k_1}{n_1}$, we have
 \[
  |\hf_t|\leq \frac{k_1}{n_1}\prod_{j=1}^\ell\binom{n_j}{k_j}\sum_{i=1}^\ell\Bigg((s-1)^2\sum_{p=1}^\ell\frac{k_p^2}{n_p}
  +6s(s-1)\frac{k_i(k_i-1)}{n_i-s-k_i+1}
\Bigg).
 \]
Since $n_i\geq 8\ell^2k_i^2s$ and  $\ell\geq 2$, we have
\[
\sum_{i=1}^\ell\Big((s-1)^2\sum_{p=1}^\ell\frac{k_p^2}{n_p}\Big)= \frac{(s-1)^2\ell^2}{8\ell^2s}\leq \frac{s-1}{8},
\]
and
\[ \sum_{i=1}^\ell\frac{6s(s-1)k_i(k_i-1)}{n_i-s-k_i+1}\leq \sum_{i=1}^\ell\frac{6s(s-1)k_i(k_i-1)}{16\ell k_i(k_i-1)s +16\ell k_is-s-k_i+1} \leq \frac{3(s-1)}{8} .
\]
Moreover, $\frac{k_i}{n_i}\leq \frac{k_1}{n_1}$. Hence, by inequality \eqref{ineq-6} we obtain that
\begin{align*}
  |\hf_t|
  &\leq\frac{(s-1)}{2}\binom{n_1-1}{k_1-1}\prod_{j=2}^\ell\binom{n_j}{k_j}\\
  &\leq(s-1)\binom{n_1-s+1}{k_1-1}\prod_{j=2}^\ell\binom{n_j}{k_j}\\
  &\leq\left(\binom{n_1}{k_1}-\binom{n_1-s+1}{k_1}\right)\prod_{j=2}^\ell\binom{n_j}{k_j},
\end{align*}
and this completes the proof.
\end{proof}

\begin{proof}[Proof of Theorem \ref{main-2}.] As the same to the proof of Theorem \ref{main-1}, we have
\[
 \min_{x_1+\dots+x_\ell=s}\prod_{i=1}^\ell\binom{n_i-x_i}{k_i} =\min_{1\leq i\leq \ell} \binom{n_i-s}{k_i}\prod_{j\neq i}\binom{n_j}{k_j}.
\]
Thus, Lemma \ref{lem:main-2} implies  the theorem.
\end{proof}

\bibliographystyle{abbrv}

\end{document}